\documentclass[a4paper,draft,oneside,12pt]{article}

\usepackage{ amsthm,amsmath,amsfonts,amssymb}
\usepackage{epsfig}
\usepackage{colordvi,helvet,multicol}
\usepackage{epic}
\usepackage{mathptmx}
\usepackage{times}
\usepackage[T1]{fontenc}
\usepackage{verbatim}
\usepackage{graphics}
\usepackage{makeidx}
\usepackage{fancyhdr}
\makeindex
\makeatletter
\long\def\@caption#1[#2]#3{%        Abbildungstext auf Fussnotengroesse
  \par
  \addcontentsline{\csname ext@#1\endcsname}{#1}%
    {\protect\numberline{\csname the#1\endcsname}{\ignorespaces #2}}%
  \begingroup
    \@parboxrestore
    \if@minipage
      \@setminipage
    \fi
   \normalsize
    \footnotesize
    \@makecaption{\csname fnum@#1\endcsname}{\ignorespaces #3}\par
  \endgroup}
\makeatother
 \theoremstyle{plain}
\newcommand{\ee}{\end{equation}}

\newtheorem{theorem}{Theorem}[section]

\newtheorem{remark}{Remark}[theorem]
\newtheorem{proposition}{Proposition}[section]
\newtheorem{corollary}{Corollary}[theorem]

\newcommand\RR{{\Bbb R}}
\newcommand\CC{{\Bbb C}}

\newcommand\NN{{\Bbb N}}

\begin{document}

\title{    $\alpha$-Amenable Hypergroups 
}
\author{Ahmadreza  Azimifard
 \date{}
 }
 \maketitle

\maketitle

%%%%%%%%%%%%%%%%%%%%%%%%%%%%%%%%%%%%%%%%%%%%%%%%%%%%%%%%%%%%%%%%%%%%%%%%%%%%

\begin{abstract}
                 Let $K$ denote a locally compact commutative hypergroup,
                 $L^1(K)$ the hypergroup algebra, and $\alpha$ a real-valued
                 hermitian character of $K$. We show that $K$ is $\alpha$-amenable
                 if and only if $L^1(K)$ is $\alpha$-left amenable.
                 We also consider   the $\alpha$-amenability of hypergroup
                 joins and polynomial hypergroups in several variables as well as a single variable.

\vspace{1.cm}

\footnotesize{
\begin{tabular}{llrl}
{\bf{  Keywords.}}  &  \multicolumn{2}{l}
{ $\alpha$-Amenable hypergroups;   }\\
&    {    Koornwinder, associated Legendre,  Pollaczek, and disc  polynomials}\\
   \end{tabular}
 \vspace{.3cm}

  {\bf  AMS  Subject Classification (2000).}{ Primary 43A62,   43A07.} {Secondary   46H20.}

}
\end{abstract}

%%%%%%%%%%%%%%%%%%%%%%%%%%%%%%%%%%%%%%%%%%%%%%%%%%%%%%%%%%%%%%%%%%%%%%%%%%%%

{\bf{Introduction}.}
                  Let $K$ denote  a locally compact commutative hypergroup,   $L^1(K)$
                  the hypergroup
                  algebra, and $\alpha$ a hermitian character of $K$.
                  It is shown in
                  \cite{f.l.s} that  $K$ is $\alpha$-amenable if
                  and only if either
                  $K$ satisfies the modified Reiter's condition of $P_1$-type
                  in
                  $\alpha$ or the maximal ideal in $L^1(K)$ generated
                  by $\alpha$ has a bounded approximate identity.
                  For instance, $K$ is always $(1-)$ amenable, and
                  if $K$ is compact
                  or $L^1(K)$ is amenable, then $K$ is $\alpha$-amenable in
                  every
                  character $\alpha$. It is worth noting, however, that there
                  do
                  exist  hypergroups which are not $\alpha(\not=1)$-amenable;
                  e.g. see
                  \cite{f.l.s, Ska92}. So,  the amenability of
                  a hypergroup in
                  a character $\alpha$ cannot in general
                  imply its amenability in
                  other characters  even if $\alpha$ is integrable,   as illustrated
                  in Section \ref{main.II}.
                  In fact,   this kind of amenability
                  of  hypergroups depends heavily
                  on the asymptotic behavior
                  of characters as well as Haar measures, as
                  demonstrated in this paper
                  and \cite{azimifard.Monath, azimi.C.Math.Rep.,  f.l.s}.

               The paper is devoted to the character amenability
               of hypergroups.
               Sections \ref{first.section} and  \ref{main.II} contain our  main results.
               First  we show that   if the character $\alpha$ is real-valued, then  $K$ is $\alpha$-amenable
               if and only if $L^1(K)$ is $\alpha$-left amenable; see  Theorem \ref{main.3}.
               We then (Theorem \ref{main.6}) consider   the $\alpha$-amenability of
               hypergroup joins.
               Section \ref{main.II} is  restricted to the
               polynomial hypergroups.  Theorem \ref{main.2}  provides
               a necessary condition for  the $\alpha$-amenability of
               hypergroups; and, subsequently we use this result to examine the $\alpha$-amenability
               of various polynomial
               hypergroups. In fact, we show
               that the majority of common examples of polynomial hypergroups are
               only $1$-amenable,  and
               Example (VI)  illustrates  just  how complicated hypergroups can be. 

               Parts of this paper are taken from author's
               dissertation at  Technische Universität München.

{\bf{Preliminaries.}} \label{preliminary}
                  Let $(K, p, \sim )$ denote a
                  locally compact commutative hypergroup with Jewett's axioms \cite{Jew75}, where
                  $p:K\times K\rightarrow M^1(K)$, $(x,y)\mapsto p(x,y)$,
                  and $\sim:K\rightarrow K$,  $x\mapsto \tilde{x}$,
                  specify  the convolution and involution on $K$ and $p{(x,y)}=p{(y,x)}$ for
                  every $x, y\in K$. Here
                  $M^1(K)$ denotes the set of all probability measures on  $K$.

                  Let us first recall    required notions here,
                  which are
                  mainly from \cite{BloHey94, Jew75}.
                  Let $C_c(K)$, $C_0(K)$, and $C^b(K)$ be the
            spaces of all  continuous functions, those which have
            compact support,
            vanishing at infinity, and bounded on $K$,  respectively. Both $C^b(K)$
            and $C_0(K)$ will be topologized by the uniform norm
            $\| \cdot  \|_\infty$, and by Riesz's theorem  $C_0(K)^\ast\cong M(K)$,
            the  space of all complex regular  Radon measures on $K$.
            The translation of $f\in C_c(K)$ at  the point
            $x\in K$, $T_xf$, is defined by $T_xf(y)=\int_K  f(t)dp{(x,y)}(t)$,
            for every $y\in K$.

              Let  $m$ denote the unique Haar measure of $K$ \cite{Spec75} and $(L^p(K), \|\cdot\|_p)$ $(p\geq 1)$
              the usual Banach space. If   $p=1$,   $(L^1(K), \|\cdot \|_1)$   is a Banach
              $\ast$-algebra where  the convolution and involution of
              $f,g\in L^1(K)$ are given by  $ f*g(x)=\int_K f(y)T_{\tilde{y}}g(x)dm(y)$ ($m$-a.e.)
              and
              $f^\ast(x)=\overline{f(\tilde x)}$
              respectively. If $K$ is discrete, then $L^1(K)$ has an identity element;
              otherwise $L^1(K)$ has a bounded approximate identity (b. a. i.), i.e.
              there exists a net $\{e_i\}_i$ of functions in $L^1(K)$ with $\|e_i\|_1\leq M$,   $M>0$,
              such that $\|f \ast e_i-f\|_1\rightarrow 0$ as
              $i\rightarrow \infty$.

              The dual of  $L^1(K)$ can be identified with
              the usual Banach space
              $L^\infty(K)$, and its structure space is  homeomorphic to the
              character
              space of $K$, i.e.
    \[ \mathfrak{X}^b(K):=\left\{
                                     \alpha\in C^b(K): \alpha(e)=1, \;p(x,y)
                                    (\alpha)=\alpha(x)\alpha(y), \; \forall\;x,y\in
                                    K\right\}\]
         equipped with the compact-open topology.
        $\mathcal{X}^b(K)$ is a locally compact Hausdorff space.
        Let  $\widehat{K} $ denote the set of all  hermitian
        characters $\alpha$ in $\mathcal{X}^b(K)$,
        i.e. $\alpha(\tilde{x})=\overline{\alpha(x)}$ for every $x\in
        K$, with   a  Plancherel measure $\pi$.
        In contrast  to the case of groups,      $\widehat{K}$
        might  not have    the dual hypergroup structure
        and might  properly contain $ \mathcal{S}=\mbox{supp }{\pi}$.

       The Fourier-Stieltjes transform of $\mu\in M(K)$, $\widehat{\mu}\in C^b(\widehat{K})$,
       is  given by  $\widehat{\mu}(\alpha):=\int_K \overline{\alpha(x)}d\mu(x)$. 
       Its restriction to   $L^1(K)$   is called the Fourier transform.
       We have  $\widehat{f}\in C_0(\widehat{K})$, for $f\in
       L^1(K)$, and $I(\alpha):=\{f\in L^1(K): \widehat{f}(\alpha)=0\}$ is
       the maximal ideal in $L^1(K)$ generated
       by $\alpha$ \cite{BonDun73}.

       $K$ is called  $\alpha$-amenable $(\alpha\in \widehat{K})$ if there exists
       $m_\alpha\in L^{\infty}(K)^{\ast}$ such that (i)
       $m_\alpha(\alpha)=1$ and (ii)
       $m_\alpha(T_xf)=\alpha(x)m_\alpha(f)$ for every
       $f\in L^\infty(K)$ and $x\in K$.  $K$ is called
         amenable if the latter holds for $\alpha=1$.

       For the sake of completeness,  we recall
       the modified Reiter's condition of $P_1$-type
       in $\alpha \in \widehat{K}$ from \cite{f.l.s} which is
       required    in Theorem \ref{main.5}. By this condition we
       shall mean  for every $ \varepsilon>0$ and
       every compact subset $C$ of $K$ there exists  $ g\in L^1(K)$ with  $\|g\|_1\leq M$
       $(M>0)$
       such that $\widehat{g}(\alpha)=1$ and $ \|T_{x}g-\alpha(x)g\|_1<\varepsilon$
       for all $x\in C.$  The condition is simply called Reiter's condition if
       $\alpha=1$ \cite{Ska92}.

         \section{$\alpha$-Left Amenability of $L^1(K)$}\label{first.section}
               Let $X$ be a   Banach $L^1(K)$-bimodule and $\alpha\in
               \widehat{K}$. Then, in a canonical way,
                the dual space $X^\ast$ is a Banach
                $L^1(K)$-bimodule. The module   $X$ is called a $\alpha$-left
               $L^1(K)$-module if the left module multiplication is given by
               $f\cdot x=\widehat{f}(\alpha)x$, for  every
                $f\in L^1(K)$ and $x\in X$. In this case,  $X^\ast$ turns out  to be a
                $\alpha$-right $L^1(K)$-bimodule as well,
                i.e.  $\varphi\cdot f=\widehat{f}(\alpha)\varphi$,  for
                every $f\in L^1(K)$ and $\varphi\in X^\ast$.

                A continuous  linear map $D:L^1(K)\rightarrow X^\ast$
                is called a derivation if $D(f\ast g)=D(f)\cdot g+f\cdot D(g)$,
                for every  $f,g\in L^1(K)$, and  an inner derivation
                if $D(f)=f\cdot \varphi-  \varphi\cdot f$,
                for some $\varphi\in X^\ast$.   The algebra $L^1(K)$ is called
                $\alpha$-left amenable if for every $\alpha$-left
               $L^1(K)$-module $X$,    every
               continuous derivation $D:L^1(K)\rightarrow X^\ast$
               is inner; and, if the latter holds for every
               Banach $L^1(K)$-bimodule $X$,
               then $L^1(K)$ is called amenable.

               As shown in \cite{f.l.s}, $K$ is $\alpha$-amenable
              if and only if either $I(\alpha)$ has a b.a.i. or  $K$ satisfies
              the modified Reiter's condition of $P_1$-type
              in $\alpha$.
              In the following theorem we   explore the connection between
               the $\alpha$-amenability of $K$ and  $\alpha$-left amenability of $L^1(K)$.

        \begin{theorem}\label{main.3}
                \emph{Let $K$ be a  hypergroup and
                $\alpha\in \widehat{K}$,   real-valued.  Then $K$ is $\alpha$-amenable if and only if
                $L^1(K)$ is $\alpha$-left amenable.}
        \end{theorem}

        \begin{proof}
        Assume $K$ to be $\alpha$-amenable, choose $X$ to be an
        arbitrary
        $\alpha$-left $L^1(K)$-module, and suppose that $D:L^1(K)\rightarrow X^\ast$ is a continuous  derivation.
               For fixed   $x\in X$  define $\Phi_x\in L^1(K)^\ast$
        by $\Phi_x(f)=D(f)(x)$  for $f\in L^1(K)$. Then for every $f,g\in L^1(K)$
  \begin{align}\notag
 \Phi_x(f\ast g)&= D(f\ast g)(x)=(f\cdot D(g))(x)+(D(f)\cdot g)(x)\\ \notag
 &= D(g)(x\cdot f)+\widehat{g}(\alpha)D(f)(x)\\ \label{eq.1}
 &=\Phi_{x\cdot f}(g)+\widehat{g}(\alpha)\Phi_x(f).
 \end{align}

                    Moreover,  $\Phi_{x+y}=\Phi_x+\Phi_y$,
                    $\Phi_{\lambda\cdot x}=\lambda \Phi_x$,   and $\|\Phi_x\|\leq \|D_\alpha\|\|x\|$
                    for $x,y \in X$ and $\lambda\in \CC$.

                    We identify  $\Phi\in L^1(K)^\ast$ with 
                    $\eta\in L^\infty(K)$ by the relation
         \[\Phi(g):=\int_Kg(t)\eta(\tilde{t})dm(t)\]
         for $g\in L^1(K)$.
         Denote by $\eta_x$ and $\eta_{x\cdot f}$ the elements of $L^\infty(K)$
         corresponding to $\Phi_x$ and $\Phi_{x\cdot f}$, respectively.
         Thus, $\eta_{x+y}=\eta_x+\eta_y$,  $\eta_{\lambda x}=\lambda \eta_x$, and
         $\|\eta_x\|_\infty \leq \|D_\alpha\|\|x\|$ for $x, y\in X$ and $\lambda\in\CC$.

         Now (\ref{eq.1}) can be rewritten  as
            \begin{equation}\label{eq.2}
                      \int_K f\ast g(t)\eta_{x}(\tilde{t})dm(t)=
                      \int_K g(t) \eta_{x\cdot f}(\tilde{t} )dm(t)+
                      \widehat{g}(\alpha)\int_K f(t)\eta_x(\tilde{t})dm(t)
             \end{equation}

             Applying Fubini's theorem and \cite[Thm.1.3.21]{BloHey94} yield
           \begin{align}\notag
           \int_K f\ast g(t) \eta_x(\tilde{t})dm(t)&=
           \int_K \left(\int_K f(y)T_{\tilde{y}}g(t)dm(y)\right)\eta_x(\tilde{t})dm(t)\\ \notag
                                                   &=
           \int_K \left( \int_K g(t)T_{y}\eta_x(\tilde{t})dm(t)\right)f(y)dm(y)\\ \notag
                                                   &=
           \int_K\left(\int_K f(y)T_{y}\eta_x(\tilde{t})dm(y)\right)g(t)dm(t)\notag
          \end{align}
and
          \begin{align}\notag
          \int_K g(t)\eta_{x\cdot f}(\tilde{t})dm(t)=&
          \int_K \left(  \int_K f(y) T_{ {y}}\eta_x(\tilde{t})dm(y)\right)g(t)dm(t)\\
                                                    &
          -\int_K g(t) \left( \alpha(\tilde{t})\int_K
          f(y)\eta_x(\tilde{y})dm(y)\right)dm(t)\label{equ.44}.
          \end{align}

          Now, by assumption there exists
           $m_\alpha\in L^\infty(K)^\ast$ such that  $m_\alpha(\alpha)=1$
           and
           $m_\alpha(T_y\eta)=\alpha(y)m_\alpha(\eta)$ for every
           $\eta\in L^\infty(K)$ and  $y\in K$. Let
         \begin{equation}\notag
         \varphi(x):=m_\alpha(\eta_x), \hspace{1.cm}x\in X.
         \end{equation}
    Then $\varphi(x+y)=\varphi(x)+\varphi(y)$,
    $\varphi(\lambda x)=\lambda\varphi(x)$
    and $|\varphi(x)|\leq \|m_\alpha\|\|D_\alpha\|\|x\|$.
    Hence $\varphi\in X^\ast$, and for $f\in L^1(K)$
    and $x\in X$ it follows
    that
         \begin{equation}\notag
         f\cdot \varphi(x)=\varphi(x\cdot f)=m_\alpha(\eta_{x\cdot
         f}).
         \end{equation}

          By Goldstein's theorem \cite{Dunford},  the functional
          $m_\alpha$ is the $w^\ast$-limit of a net of functions  $g\in L^1(K)$,
          therefore
           from (\ref{equ.44}) we obtain that
         \begin{equation}\notag
         m_\alpha(\eta_{x\cdot f})=\int_K f(y)
         m_\alpha(T_{y}\eta_x)dm(y)-m_\alpha(\alpha)\Phi_x(f),
         \end{equation}
            and hence
          \begin{equation}\notag
               \Phi_x(f)=\widehat{f}(\alpha)m_\alpha(\eta_x)-m_\alpha(\eta_{x\cdot f})
               \hspace{.5cm} \mbox{for } f\in L^1(K), x\in X.
          \end{equation}
   That means $D(f)(x)=\varphi\cdot f(x)-f\cdot \varphi(x)$, thence  $D$ is an inner derivation,
    and this gives the $\alpha$-left amenability of $L^1(K)$.\\

        To prove the converse of  the theorem  we follow the method 
        in \cite[p.239]{BonDun73}. Assume  $L^1(K)$ to be $\alpha$-left
        amenable and consider $\alpha$-left $L^1(K)$-module $L^\infty(K)$
        with the module multiplications
        $f\cdot \varphi=\widehat{f}(\alpha)\varphi$
        and $\varphi\cdot f=f\ast\varphi$,
        for every  $f\in L^1(K)$ and
        $\varphi\in L^\infty(K)$.
        Since $\alpha\cdot f=f\ast \alpha=\widehat{f}(\alpha)\alpha$,
        $\CC\alpha$ is a closed $L^1(K)$-submodule of $L^\infty(K)$.
        Hence, $L^\infty(K)=X\oplus \CC\alpha$ where $X$ is also
        a closed $L^1(K)$-submodule of $L^\infty(K).$ Choose
        $\nu\in L^\infty(K)^\ast$ such that $\nu(\alpha)=1$, and define
        $\delta:L^1(K)\rightarrow L^\infty(K)^\ast,$ $\delta(f)=f\cdot \nu-\nu\cdot f$,
        for $f\in L^1(K).$
        Then
\begin{align}\notag
            \delta(f)(\alpha)&=f\cdot \nu(\alpha)-\nu\cdot f(\alpha)\\ \notag
                                &=\nu(\alpha\cdot f)-\nu(f\cdot \alpha)\\ \notag
                                &=\nu(f\ast \alpha)- \widehat{f}(\alpha)\nu(\alpha)\\\notag
                                &=\widehat{f}(\alpha)\nu(\alpha)-\widehat{f}(\alpha)\nu(\alpha)=0,
                                \hspace{1.cm}\mbox{ for every } f\in L^1(K).
\end{align}

        That means
        $\delta(f)\in \left(\CC\alpha\right)^{\bot}\subset L^\infty(K)^\ast$.
        Let  $P:L^\infty(K)\rightarrow X$ denote the projection onto
        $X$  and  $P^\ast:X^\ast\rightarrow L^\infty(K)^\ast$
         the adjoint operator. $P^\ast$ is an injective $L^1(K)$-bimodule homomorphism;   it
         follows that $(\CC\alpha)^\bot=(Ker P)^\bot=(P^\ast(X^\ast)_\bot)^\bot=P^\ast(X^\ast)$.
         Hence,  for each $f\in L^1(K)$ there exists
         $D(f)\in X^\ast$ such that $P^\ast D(f)=\delta(f).$
         Since $\delta$ is a continuous  derivation  on $L^1(K)$,
         the map $D:L^1(K)\rightarrow X^\ast$
         is a continuous  derivation as well. By assumption $D$ is
         inner, that is, there exists  $\psi\in X^\ast$ such that
         $D(f)=f\cdot \psi- \psi\cdot f$ for all
         $f\in L^1(K)$. Define $m_\alpha:=\nu-P^\ast \psi$.
         Then $m_\alpha(\alpha)=\nu(\alpha)-P^\ast \psi(\alpha)=1-\psi(P\alpha)=1$
and
\[
        f\cdot (P^\ast \psi)(\varphi)=P^\ast \psi(\varphi\cdot f)=
        \psi(P(\varphi\cdot f))=f\cdot \psi(P\varphi)=
        P^\ast (f\cdot\psi)(\varphi)\hspace{.3cm}(\varphi\in X).
\]
 Similarly
            $(P^\ast \psi)\cdot f(\varphi)= P^\ast (\psi\cdot f)(\varphi)$, thus
 \begin{align}\notag
            f\cdot P^\ast \psi-P^\ast \psi\cdot f&=P^\ast (f\cdot \psi-\psi\cdot f) \\ \notag
                                                    &=P^\ast Df=\delta(f)=f\cdot \nu-\nu\cdot f.\notag
 \end{align}
                Hence,  $f\cdot m_\alpha=f\cdot \nu-f\cdot P^\ast\nu=\nu\cdot f-P^\ast \psi\cdot f=m_\alpha\cdot f$.
 This  means
         \[
                    m_\alpha(\varphi\ast f)=m_\alpha(\varphi\cdot f)=f\cdot m_\alpha (\varphi)
                    = m_\alpha\cdot f(\varphi)=m_\alpha(f\cdot\varphi)=\widehat{f}(\alpha) m_\alpha (\varphi),
         \]
and hence 
    \[
            m_\alpha(T_x\varphi)=m_\alpha(\delta_{\tilde{x}}\ast \varphi)
                =\overline{\alpha(\tilde{x})}
                m_\alpha(\varphi)=\alpha(x)m_\alpha(\varphi),
    \]
 for every $\varphi\in L^\infty(K)$ and $ x\in K $, giving the $\alpha$-amenability of $K$.

 \end{proof}

                    The previous theorem combined with \cite{f.l.s}  yield     Johnson-Reiter's condition for
                    hypergroups,  in the $\alpha$-setting,  which reads  as follows.
\begin{theorem}\label{main.5}
    \emph{Let $K$ be a  hypergroup and $\alpha\in \widehat{K}$,
    real-valued. Then the following  statements are equivalent:}
    \begin{itemize}
            \item[\emph{(i)}]\emph{$K$ is $\alpha$-amenable.}
            \item [\emph{(ii)}]\emph{$L^1(K)$ is $\alpha$-left amenable.}
            \item  [\emph{(iii)}]\emph{ $I(\alpha)$ has a b.a.i.}
            \item [\emph{(iv)}]  \emph{$K$ satisfies the modified $P_1$-condition  in   $\alpha$.}
 \end{itemize}
\end{theorem}

\hspace{1.cm}

\begin{corollary}
\emph{
                If $K$ is $\alpha$-amenable, then  every functional
                 $D:L^1(K)\rightarrow\CC$ such that $D(f\ast g)=\widehat{f}(\alpha)D(g)+\widehat{g}(\alpha)D(f)$,
                 $f, g\in L^1(K)$,  is zero (see \cite[5.2]{azimifard.Monath}).
                     The converse, however,  is in general not true;
                     see Example (II) or  \cite[5.5]{azimifard.Monath}.
                     The functional $D$ is called a $\alpha$-derivation.}
\end{corollary}
%%%%%%%%%%%%%%%%%%%%%%%%%%%%%%%%%%%%%%%%%%%%%%%%%%%%%%%%%%%%%%%%%%%%%%%%%%%%%%%%%%%%%%%%

\begin{remark}\label{cor.1}
\begin{itemize}
\item[\emph{(i)}] \emph{ If $\alpha\in L^1(K)\cap L^2(K)$,   then
                   \[m_\alpha(f):=\frac{1}{\|\alpha\|_2^2}\int_K f(x)\overline{\alpha(x)}dm(x),\hspace{1.cm}f\in L^\infty(K),\]
                    is a  $\alpha$-mean on $L^\infty(K)$.
                     For example, if $K$ is a hypergroup of compact type
                     \cite{Frankcompact}, the functional $m_\alpha$ is an $\alpha$-mean
                     on $L^\infty(K)$ for every $\alpha\in \widehat{K}\setminus{\{1\}}$;
                     this holds also for $\alpha=1$ if $K$ is compact. We  note that 
                      the $\alpha$-means $m_\alpha$,  given as above, are unique
                     \cite{azimi.C.Math.Rep.}.}
                                                                      
                                \item[\emph{(ii)}]  \emph{
                                Observe that $\widehat{K}$ might  contain some positive  characters
                                $\alpha\not=1$ in which case $K$ is $\alpha$-amenable; see  Example (VI).}

\end{itemize}

\end{remark}

                  Our next topic is about   the $\alpha$-amenability of
                  hypergroup joins, and     Theorem \ref{main.6}
                  generalizes \cite[3.12 ]{Ska92} to the $\alpha$-setting. For the
                  sake of convenience, we first recall the definition of
                  hypergroup joins  and some known facts about  their dual spaces.
                  Let $(H, \ast )$ be a compact hypergroup with  a normalized Haar measure $m_H$, $(J, \cdot)$ a
                  discrete hypergroup  with a Haar measure $m_J$, and suppose that    $H\cap J=\{e\}$, where $e$
                  is the identity of both hypergroups.
                  The hypergroup joins $(H\vee J, \odot)$ is the set $H\cup J$
                  with  the unique topology for which $H$ and $J$
                  are closed subspace of,   where the   convolution $\odot$ is
                  defined as follows:
\begin{enumerate}
                 \item  $\varepsilon_x\odot \varepsilon_y$ agrees with that on $H$ if $x,y\in
                        H$,
                \item $\varepsilon_x\odot \varepsilon_y=\varepsilon_x\cdot \varepsilon_y$ if
                        $x,y\in J$, $x\not=\tilde{y}$,
                \item $\varepsilon_x\odot \varepsilon_y=\varepsilon_y=\varepsilon_y \odot \varepsilon_x$ if  $x\in H$, $y\in
                        J\setminus{\{e\}}$, and
                \item if $y\in J$ and  $y\not=e$,
        \begin{equation}\notag
            \varepsilon_{\tilde{y}}\odot \varepsilon_y= c_em_H+\sum_{w\in J\setminus{\{e\}}}c_w\varepsilon_w
        \end{equation}
                    where $\varepsilon_{\tilde{y}}\cdot  \varepsilon_y=\sum_{w\in J} c_w\varepsilon_w$,
                    $c_w\geq 0$,
                    only finitely many $c_w$ are nonzero,  and $\sum_{w\in J}c_w\varepsilon_w=1$.
\end{enumerate}
If  $m_J(\{e\})=1$,
                then $m_K:=m_H+1_{J\setminus{\{e\}}}m_J$\; is a Haar measure
                for
                $K$.
                Observe that $K//H\cong J$ and $H$ is a subhypergroup of $K=H\vee J$ but that $J$ is not unless
                either $H$ or $J$ is trivial \cite{vrem.joint}.  
                  As proved in  \cite[p. 119]{BloHey94},
                $\widehat{K}= \widehat{H}\cup \widehat{J} $,
                where $\widehat{H}\cap \widehat{J}=\{1\}$. The latter holds
                in the sense of hypergroup
                isomorphism, $\widehat{K}\cong \widehat{H}\vee\widehat{J}$,
                if $H$ and $J$ are strong hypergroups. In
                this case $K$ is   a
                strong hypergroup as well.

 \begin{theorem}\label{main.6}
            \emph{Let $K$ be as above, $|J|\geq 2$, and  $\alpha\in\widehat{J}$.
             Then  $J$ is $\alpha$-amenable if and only if  $K$ is
             $\alpha$-amenable. Moreover,
             if $H$ and $J$ are strong hypergroups,
             then $\widehat{H}$ is $\beta$-amenable if and only if
             $\widehat{K}$ is
             $\beta$-amenable ($\beta\in \widehat{\widehat{H}}$). }
\end{theorem}

\begin{proof}
                    Let $x\in J^\ast:=J\setminus{\{e\}}$.  By \cite[3.15]{Ska92} for
                    $f\in L^\infty(K)$, we have
        \begin{equation}\label{equ.4}
                T_xf=T_x(f|_{J^\ast})+T_x(1_H)\int_H f(t)dm_H(t).
        \end{equation}
                Now, take $\alpha\in \widehat{J}$ and   assume $J$ to be  $\alpha$-amenable.
                Then there exists $m_\alpha:\ell^\infty(J)\rightarrow \CC$ such that
                $m_\alpha(\alpha)=1$ and $m_\alpha(T_xf)=\alpha(x)m_\alpha(f)$,  for
                all $f\in \ell^\infty(J)$ and $x\in J$. The character $\alpha$ can
                be extended to $K$ by letting    $\gamma(x):=1$ for all   $x\in H$.
                Define
            \[M_\gamma:L^\infty(K)\rightarrow \CC, \hspace{.5cm}
                M_\gamma(f):=m_\alpha(f|_{J^\ast}),\hspace{.5cm}f\in L^\infty(K).\]
            We have
            $M_\gamma(\gamma)=m_\alpha(\gamma|_{J^\ast})=m_\alpha(\alpha)=1$, and 
             (\ref{equ.4}) implies that
    \begin{align}\notag
        M_\gamma(T_xf)&=M_\gamma(T_x(f|_{J^\ast}))+M_\gamma
        (T_x(1_H))\int_Hf(t)dm_H(t)\\ \label{equ.6}
        &=m_\alpha(T_x(f|_{J^\ast}))=\alpha(x)m_\alpha(f|_{J^\ast})=\gamma(x)M_\gamma(f),\hspace{.4cm}\mbox{
        for all } f\in L^\infty(K),  x\in J^\ast.
    \end{align}
            Since $\gamma|_H=1$ and  $(T_xf)|_{J^\ast}=f|_{J^\ast}$ for
             $x\in H$, the equality (\ref{equ.6})
             is valid for all $x\in K$. Therefore, $K$ is $\gamma$-amenable.

         To prove the converse, let  $\gamma\in \widehat{K}$
         and   assume $K$ to  be $\gamma$-amenable.
         If $\gamma|_H=1$, then by $\widehat{K}= \widehat{H}\cup \widehat{J} $
         we have $\gamma\in \widehat{J}$.
         Define
            \[m_\gamma:\ell^\infty(J)\rightarrow \CC, \hspace{.5cm} m_\gamma(f):=M_\gamma(f),  \hspace{.5cm}f\in L^\infty(K),\]
        where $M_\gamma$ is a $\gamma$-mean  on $L^\infty(K)$. Obviously
        $m_\gamma(\gamma)=1$,  and
            \[m_\gamma(T_xf)=M_\gamma(T_xf)=M_\gamma(T_x(f|_{J^\ast}))=\gamma(x)M_\gamma(f)=\gamma(x)m_\gamma(f), \]
        for all $f\in \ell^\infty(J)$ and $x\in J^\ast.$
        If $\gamma(x)\not=1$ for some $x\in H$, then  $\gamma$
        is a nontrivial  character of $H$ and
        $\widehat{K}= \widehat{H}\cup \widehat{J} $
         implies that $\gamma|_J=1$. Since
         $J$ is  commutative, $J$ must be 
         amenable \cite{Ska92},  and
         the amenability of  $H$ in
         $\gamma$ follows from Remark \ref{cor.1}.

        The proof of the second part can be obtained  from the
        first part and the fact that $\widehat{K}\cong \widehat{H}\vee \widehat{J}$ if $H$
        and $J$ are strong hypergroups.

\end{proof}

\section{$\alpha$-Amenability of polynomial hypergroups}\label{main.II}
             In this section we restrict our discussion to the polynomial
             hypergroups. First  we consider polynomial
             hypergroups in several variables which have been
             already studied by several authors
 (e.g. see  \cite{Koornwinder.1,zeuner.several}).
            The translation operators of these hypergroups seem to be   complicated, and  the study of
            their character amenability  via the modified Reiter's condition, in
            contrast  to the one variable case \cite{f.l.s},
            may
            require  sophisticated calculations.
             In Theorem \ref{main.2}, however, we provide a  necessary condition
            to the $\alpha$-amenability of these hypergroups.  Hence we point out that the  majority of
            common examples of polynomial hypergroups do not satisfy this condition. \\

            Let $\{P_{\bf{n}}\}_{{\bf{n}}\in \mathcal{K}}$ be a set of
            orthogonal polynomials on $\CC^d$ with respect to a measure $\pi\in
            M^1(\CC^d)$ such that $P_{\bf{n}}(u)=1$ for some $u\in \CC^d$, where
            $\mathcal{K}:=\NN_0^m $ with
             the discrete topology, $m,d\in \NN$, and  $\NN_0:=\NN\cup \{0\}$.
             Assume  $\mathcal{P}_n$ denotes
            the set of all polynomials $ P' \in \CC[z_1,z_2,...z_d]$ with degree
            less or equal than $n$ and $\mathcal{K}_n:=\{{\bf{n}}\in
            \mathcal{K}: P_{\bf{n}}\in \mathcal{P}_n\}$. Suppose that for every
            $n\in \NN$ the set $\{P_{\bf{n}}: {\bf{n}}\in \mathcal{K}_n\}$ is a
            basis of $\mathcal{P}_n$,  and for every ${\bf{n}},{\bf{m}}\in \mathcal{K}$
            the product $P_{\bf{n}}\cdot P_{\bf{m}}$ admits the unique non-negative
            linearization formula, i.e.
\begin{equation}\label{recursion.1}
        P_{\bf{n}}\cdot P_{\bf{m}}:=\sum_{{\bf{t}}\in \mathcal{K}}
        g({\bf{n}},{\bf{m}},{\bf{t}}) P_{\bf{t}}
\end{equation}
        where $g({\bf{n}},{\bf{m}},{\bf{t}})\geq 0$.
        Assume further that there exists
        a homeomorphism  ${\bf{n}}\rightarrow \tilde{\bf{n}}$ on $\mathcal{K}$ such that
         $P_{\tilde{\bf{n}}}=\overline{P_{\bf{n}}}$ for every  ${\bf{n}}\in \mathcal{K}$.
         In this case $\mathcal{K}$   with  the convolution of two point measures defined by
         $\varepsilon_{\bf{n}}\ast \varepsilon_{\bf{m}}(\varepsilon_{\bf{t}}):=
         p({\bf{n}},{\bf{m}})(\varepsilon_{\bf{t}}):=
         g ( {\bf{n}},{\bf{m}}, {\bf{t}})$ is  a
         hypergroup which  is called a polynomial hypergroup in $d$
         variables. The  hypergroup $\mathcal{K}$
         is obviously commutative and the identity element $e$ is the constant
         polynomial $P_0 \equiv 1.$ The character space $\widehat{\mathcal{K}}$ can
         be identified with  the set
         $ \{  {\bf{x}}\in \CC^d: | \alpha_{\bf{x}}({\bf{n}})|\leq
         1, \alpha_{\bf{x}}(\tilde{\bf{n}}) =\overline{\alpha_{\bf{x}}({\bf{n}})} \hspace{.2cm} \forall {\bf{n}}\in \mathcal{K} \}$,
         where
            $\alpha_{\bf{x}}({\bf{n}}):=P_{{\bf{n}}}({\bf{x}})$
            for ${\bf{x}}\in \CC^d$ and ${\bf{n}}\in \mathcal{K}$.
            For more
            on  polynomial
            hypergroups in several variables we refer
             the reader  to e.g. \cite{BloHey94,Koornwinder.1, zeuner.several}.

\begin{theorem}\label{main.2}
             \emph{ Let $\{P_{\bf{n}}({\bf{x}})\}_{{\bf{n}}\in \mathcal{K}}$
             define a    polynomial hypergroup in $d$ variables  on
             $\mathcal{K}:=\NN_0^m$ and $\alpha_{\bf{x}}\in\widehat{\mathcal{K}}$
             with $\pi(\{\alpha_{\bf{x}}\})=0$. If
             $\alpha_{\bf{x}}\in C_0(\mathcal{K})$,
             then $\mathcal{K}$ is not
             $\alpha_{\bf{x}}$-amenable. }
\end{theorem}

\begin{proof}

                     Assume to the contrary  that $\mathcal{K}$ is $\alpha_{\bf{x}}$-amenable
                     and $m_{\alpha_{\bf{x}}}$  is a $\alpha_{\bf{x}}$-mean on
                     $\ell^\infty(\mathcal{K})$. Due  to
\begin{equation}\notag
               T_{\bf{n}}\varepsilon_{\bf{0}}({\bf{m}})=
               \sum_{\bf{t} \in\mathcal{K}} \varepsilon_{\bf{0}}({\bf{t}})p({\bf{n}},
               {\bf{m}})({\bf{t}})= p({\bf{n}},
               {\bf{m}})({\bf{0}})\varepsilon_{\tilde{{\bf{n}}}}({\bf{m}})
               =\frac{1}{h( \bf{n} )}\varepsilon_{\tilde{{\bf{n}}}}(\bf{m}),
\end{equation}
                    we have   $T_{\bf{n}}\varepsilon_{\bf{0}}=\frac{1}{h({\bf{n}} )}\varepsilon_{ \tilde{{\bf{n}}}}$
                    for
                    every ${\bf{n}}\in \mathcal{K}$. Therefore,
      \begin{equation}\label{eq.5}
                 m_{\alpha_{\bf{x}}}\left(\varepsilon_{ \tilde{{\bf{n}}}}\right)=
                 h( {\bf{n}} ) m_{\alpha_{\bf{x}}}(T_{\bf{n}} \varepsilon_{\bf{0}} )=
                 h({\bf{n}})\alpha_{\bf{x}}({\bf{n}})
                 m_{\alpha_{\bf{x}}}(\varepsilon_{\bf{0}}).
     \end{equation}

            Let $M>0$ be a bound for $ m_{\alpha_{\bf{x}}}$
            and 
            $\xi_{\bf{n}}=\frac{\overline{P_{\bf{n}}({\bf{x}})}}{|P_{\bf{n}}({\bf{x}})|}$ for $P_{\bf{n}}({\bf{x}})\not=0$. Then by
            the linearity of $m_{\alpha_{\bf{x}}}$ and
             (\ref{eq.5})   we have
        \begin{align}\notag
      M \geq | m_{\alpha_{\bf{x}}}(\sum_{\bf{n} \in \mathcal{M}} \xi_{\bf{n}}
      \varepsilon_{\tilde{\bf{n}}} )  | &=
      | \sum_{\bf{n}\in \mathcal{M}}\xi_{\bf{n}}
      m_{\alpha_{\bf{x}}}(\varepsilon_{\tilde{\bf{n}}} ) |=
      | \sum_{\bf{n}\in \mathcal{M}} |P_{\bf{n}}({\bf{x}})|
      h({\bf{n}})
      m_{\alpha_{\bf{x}}}(\varepsilon_{\bf{0}})|
      \\ \notag
      &
      \geq\sum_{\bf{n}\in \mathcal{M}}
      |P_{\bf{n}}({\bf{x}})|^2
      h({\bf{n}})
      |m_{\alpha_{\bf{x}}}
      (\varepsilon_{\bf{0}})|,
    \end{align}
    where   $\mathcal{M}$ is an arbitrary finite subset of $\mathcal{K}$. 
    If $ m_{\alpha_{\bf{x}}}(\varepsilon_{\bf{0}})\not=0$, then the  provious inequalities show that 
    $\alpha_{\bf{x}}\in \ell^1(\mathcal{K})\cap \ell^2(\mathcal{K})$, hence $\pi(\alpha_{\bf{x}})>0$ (see \cite[Proposition 2.5.1]{BloHey94})
     which is a constradiction. 
                     If we now  define  $\{\alpha_{\bf{x}}^{\bf{m}}\}_{{\bf{m}}\in\mathcal{K}}$ by
\begin{equation}
      \alpha_{\bf{x}}^{\bf{m}}(\bf{n}):=\left\{
      \begin{array}{l
      @{\quad{\mbox{ }}\quad}
               l}
     0 & n_i<m_i\hspace{.2cm}(1\leq  i \leq d), \\
      \alpha_{\bf{x}}({\bf{n}})&   \mbox {other,} \\
   \end{array}
  \right.
    \end{equation}
              then
              $\alpha_{\bf{x}}({\bf{n}})=(P_{\bf{n}}({\bf{x}}))_{{\bf{n}}\in \mathcal{K}}$
              can be written as follows
              \[\alpha_{\bf{x}}=\sum_{0\leq t_i\leq m_i} \varepsilon_{\bf{t}}P_{{\bf{t}}}({\bf{x}})
              +\alpha_{{\bf{x}}}^{\bf{m}}.\]
         Hence,
          \[m_{\alpha_{\bf{x}}}(\alpha_{\bf{x}})=
             \sum_{0\leq t_i\leq m_i} m_{\alpha_{\bf{x}}}(\varepsilon_{\bf{t}})P_{ {\bf{t}}}({\bf{x}})
             +m_{\alpha_{\bf{x}}}(\alpha_{\bf{x}}^{\bf{m}})
             \]
           which implies that
  \begin{equation}\label{1.equa}
           \left| m_{\alpha_{\bf{x}}}(\alpha_{\bf{x}})\right| =
           \left| m_{\alpha_{\bf{x}}}(\alpha_{\bf{x}}^{\bf{m}})\right| \leq
            M \|\alpha_{\bf{x}}^{\bf{m}} \|.\notag
            %\sup_{{\bf{n}}} |P_{\bf{n}}({\bf{x}})|.
  \end{equation}
            The latter    shows that if $\alpha_{\bf{x}}\in C_0(\mathcal{K})$,   then
             $\alpha_{\bf{x}}^{\bf{m}}\in C_0(\mathcal{K})$ for all $m\in \mathcal{K}$, hence
             $m_{\alpha_{\bf{x}}}(\alpha_{\bf{x}})=0$
             which is a contradiction .

\end{proof}

\begin{remark}\label{remark.2}
\emph{
\begin{enumerate}
    \item
         Observe   that in the preceding theorem neither of the assumptions
         $\pi(\{\alpha_{\bf{x}}\})=0$ nor
         $\alpha_{\bf{x}}\in C_0(\mathcal{K})$
         can be omitted.
         For example, a hypergroup of compact type is  $\alpha$-amenable in
         every character $\alpha$  while
         $1$ is the only character in  $\widehat{K}$ with the vanishing
         Plancherel measure \cite{Frankcompact, Ska92};   see also  Example (VI).
   \item   Theorem \ref{main.2}  is known for $m=d=1$  in
  \cite{f.l.s}.\\
\end{enumerate}
  }
\end{remark}

                We continue the section by examining  the $\alpha$-amenability
                of various polynomial hypergroups.   Let us first  start with
                polynomial hypergroups in
                two variables which have been extensively studied by T.
                H. Koornwinder in \cite{Koornwinder.1}.

\begin{enumerate}
  \item[(I)]
            {\bf{Koornwinder Class $V$ hypergroups:}}
            In this case $\mathcal{K}:=\{(n,k)\in \NN_0^2: n\geq k\}$ and
            the characters are given by
            \[P_{\bf{n}}(x,y):=P_{(n,k)}^{\alpha,\beta,\gamma,\eta}(x,y):=
            P_{n-k}^{(\alpha,\beta)}(x) P_{k}^{(\gamma,\eta)}(y),\hspace{.5cm}{\bf{n}}=(n,k), \]
            where $P_{n}^{(\alpha,\beta)}$ denote  the Jacobi polynomials,  $(\alpha,\beta)$,
            $(\gamma, \eta)\in V$, $P_{\bf{n}}^{\alpha,\beta,\gamma,\eta}(1,1)=1$,
              and
    \begin{align}\notag
                       V:=\{(\alpha,\beta)\in \RR^2:&\;\alpha\geq \beta>-1, (\alpha+\beta+1)
                                                      (\alpha+\beta+4)^2(\alpha+\beta+6)\\ \notag
                                                    &
                                                    \geq (\alpha-\beta)^2\cdot
                                                    (\alpha^2-2\alpha\beta+\beta^2-5\alpha-5\beta-30)\}.
    \end{align}
   The support of the Plancherel  measure
     $d\pi(x,y)=(1-x)^\alpha(1+x)^\beta(1-y)^\gamma(1+y)^\eta dx dy$  is
     $$D:=\{(x,y)| -1\leq x\leq 1, -1\leq y\leq 1\}.$$
     Since  $ | P_n^{(\alpha,\beta)}(y)|=\mathcal{O}(n^{-\alpha-\frac{1}{2}})$ as $n\rightarrow \infty $ \cite{ismail},
     we have
 % \cite{ismail}.
 \[|P_{(n,n)}^{\alpha,\beta,\gamma,\eta}(x,y)|= |P_{n}^{(\gamma,\eta)}(y)|\rightarrow 0\hspace{1.cm}(n\rightarrow \infty)\]
 when $(x,y)\in [-1,1]\times (-1,1)$ and    $\alpha$,  $\eta>-\frac{1}{2}$. So, from
  Theorem \ref{main.2} it follows that $\mathcal{K}$ is not $\alpha_{(x,y)}$-amenable.

    For $(x,y)\in \{(-1, 1), (1, -1)(-1, -1)\}$, if $\alpha>\beta$ and
   $\gamma>\eta$,   $\alpha=\beta$ and $\gamma>\eta$, or $\alpha>\beta$  and $\gamma=\eta$
   since
   \[
        P_n^{(\alpha,\beta)}(-1)=(-1)^n
        \left(
               \begin{array}{c}
                      n+\beta \\
                      n  \\
               \end{array}
        \right)
      {\big{/}}
               \left(
                  \begin{array}{c}
                     n+\alpha \\
                     n \\
                 \end{array}
        \right),
 \]
              we have $|P_{(2n,n)}^{\alpha,\beta,\gamma,\eta}(x,y)|\rightarrow 0$
              as
              $n\rightarrow \infty,$ hence $\mathcal{K}$ is not $\alpha_{(x,y)}$-amenable.

              The hypergroup $\mathcal{K}$  is, in fact,  the product of two Jacobi polynomial
              hypergroups with parameters $(\alpha,\beta)$ and $(\gamma,\eta)$ on  $\NN_0$
              \cite{zeuner.several}.
              Theorem \ref{main.2} combined with \cite{Wol84} implies
              that $\ell^1(\NN_0)$ is amenable if
              and only  if $\alpha=\beta=\gamma=\eta=-\frac{1}{2}$.
              Thus, since  $\ell^1(\mathcal{K})\cong
              \ell^1(\NN_0)\otimes_p\ell^1(\NN_0)$,
              the algebra $\ell^1(\mathcal{K})$ is amenable  and   its  maximal ideals have
              b.a.i.; see \cite{BonDun73,Johnson}.
              Consequently,   Theorem \ref{main.5} results in the
              $\alpha_{(x,y)}$-amenability of
              $\mathcal{K}$  for   $(x,y)\in D$ and  $x,y= \pm 1$.\\

\begin{remark}
\emph{
\begin{enumerate}
               \item[(i)] Let $(x, y_0)\in [-1,1]\times [-1, 1]$
                be as above fixed.  For $ \gamma>\frac{1}{2}$ one
                can show that the usual derivation of the Fourier
                transform gives a rise to  a nonzero  bounded $\alpha_{(x,y_0)}$-derivation
                on $\ell^1(\mathcal{K})$. So,  it follows from
                Remark \ref{remark.2} that  $\mathcal{K}$ is not
                $\alpha_{(x,y_0)}$-amenable and $\{\alpha_{(x,y_0)}\}$ is not a spectral set.
               \item[(ii)]
               Similar to the previous case,  one can show that   hypergroups
               of Koornwinder class III, VI, and some
               related hypergroups in two variables  which  are mentioned in
               \cite[3.1.16-20]{BloHey94} are not $\alpha_{\bf{x}}$-amenable
               if $\alpha_{\bf{x}}\not=1$.
\end{enumerate}
 }
\end{remark}

\item[(II)] {\bf{Disc Polynomial Hypergroups}:}\label{disck}
 For $\alpha'\geq 0$ the disc polynomials

     \[
        P_{m,n}^{\alpha'} (z, \bar z)
            =\left\{
    \begin{array}{ll}
       P_n^{(\alpha', m-n) }(2z\bar z-1)z^{m-n}, & \hbox{for $m\geq n$,} \\
       P_m^{(\alpha', n-m) }(2z\bar z-1)z^{n-m}, & \hbox{for $n\geq m$,}
    \end{array}
           \right.
    \]
       induce a hypergroup structure on $\mathcal{K}:=\NN_0^2$. The support
       of the Plancherel measure
       with the density
       $(z_1, z_2)\rightarrow c_{\alpha'} (1-|z_1|^2)^{\alpha'} $ is
       $\mathcal{D}:=\{ (z_1, z_2)\in \CC^2:\; z_2=\bar z_1, \; |z_1|<1\}$.
           From Theorem  \ref{main.2} and
     \begin{align}\notag
     P_{n,n}^{\alpha'}(z, \bar z)&= P_n^{(\alpha', 0) }(2z\bar z-1)= P_n^{(\alpha', 0) }(2|z|^2-1)\\\notag
                                 &=
     \mathcal{O}(n^{-\alpha'-1/2})\hspace{1.cm}(z\in
     \mathcal{D}),
     \end{align}
     as $n\rightarrow \infty$, we infer that $\mathcal{K}$ is  $\alpha_z$-amenable if
      and only
      if $\alpha_z=1$.
      Observe that $\mathcal{H}:= \{(n, n):\; n\in \NN_0\}$ is a supernormal subhypergroup of
      $\mathcal{K}$ which is
      isomorphic to the Jacobi hypergroup
      with the character set
      $\{P_{n,n}^{\alpha'}(x) \}_{n\in \NN_0}$. In
      this case we see also
      that $\mathcal{H}$ is $\alpha_x$-amenable
      if and only if $\alpha_x=1$
      despite the fact that for
      every $x\in (-1,1)$ the singleton  $\{\alpha_x\}$ is a
      spectral for $\mathcal{H}$
      if
      $\alpha' <\frac{1}{2}$; see \cite{Wol84}.
      In other words, if $\alpha' <\frac{1}{2}$
      then every
      bounded $\alpha_x$-derivation on $\ell^1(\mathcal{H})$
      is zero, however $\mathcal{H}$
      is only 1-amenable.\\
\end{enumerate}

               In the rest of the section we   deal with  the polynomial
               hypergroups in one variable,  i.e.  the system  $\{P_{\bf{n}}\}_{{\bf{n}}\in \mathcal{K}}$
               consists of polynomials of one variable and the index set $\mathcal{K}$
               is $\NN_0$. The linearization formula  in (\ref{recursion.1}) can be
               expressed in the three term recursion formula
  \begin{equation}\label{r.10}
                P_1(x)P_n(x)=a_nP_{n+1}(x)+b_nP_n(x)+ c_nP_{n-1}(x),
  \end{equation}
               for $n\in \NN$   and $P_0(x)=1$, and we take $P_n(1)=1$, $P_1(x)=\frac{1}{a_0}(x-b_0)$
               with $a_n>0$,  $b_n\in \RR$, and  $c_{n+1}>0$ for all $n\in
               \NN_0$. The existence of the orthogonality measure  is due
               to Favard's theorem \cite{ismail} and applying it to the
               relation (\ref{r.10}) results in  $a_n+b_n+c_n=1$
               and $a_0+b_0=1$. The identity map defines
               an involution to these hypergroups and  their Haar weights are given
               by $h(0)=1$ and $h(n)=\left(\int_\RR P_n^2(x)d\pi(x)\right)^{-\frac{1}{2}}$
               $(n\geq 1)$ \cite[Theorem.1.3.26]{BloHey94}. We  consider the $\alpha$-amenability
               of following polynomial hypergroups.

\begin{enumerate}
    \item[(III)]
       {\bf{Associated Legendre hypergroups}}:
       For $\nu\in \RR_0$, let
       $\gamma_n:=\frac{(\nu+1)_n}{2^n(\nu+\frac{1}{2})_n}\left(1+\sum_{k=1}^n \frac{\nu}{k+\nu}\right)$,
       $a_n:=\frac{\gamma_{n+1}}{\gamma_{n}}$,
       $b_n:=0$, and $c_n:=1-a_n$ if  $n\geq1$ and  $\gamma_0=1$.
       The  polynomial $P_n$ associated to the sequences
       $(a_n)_{n\geq 1},(b_n)_{n\geq 1},(c_n)_{n\geq 1}$ in the
       recursion formula (\ref{r.10})
       is the $n$-th associated Legendre polynomial with
       parameter $\nu$.
       The Haar weights of the induced hypergroup on $\NN_0$ are given
       by $h(0)=1$ and
       $h(n)=\frac{2\nu+2n+1}{2\nu+1}\left( 1+ \sum_{k=1}^n\frac{\nu}{(k+\nu)^2}\right)^2 $,  $n\geq 1$,
       and the support of the Plancherel measure can be
       identified with  $[-1,1]$;
       see \cite{BloHey94}.
       If $x\in (-1,1)$,
       $\pi(\{\alpha_x\})=0$ and $\alpha_x\in C_0(\NN_0)$, so  it follows from
       Theorem \ref{main.2}
       that $\NN_0$ is $\alpha_x$-amenable if and only if $\alpha_x=1$.

    \item[(IV)] {\bf{Pollaczek polynomials hypergroup}}: The Pollaczek polynomials
        $\{P^{(\eta, \mu)}_n\}_{n\in \NN_0}$ depending
        on the parameters  $\eta\geq 0$,  $\mu>0$ or $-\frac{1}{2} <\eta<0$ and
        $0\leq \mu <\eta+\frac{1}{2}$ induce a hypergroup structure on $\NN_0$ \cite{lasser.lagure}.
        The Haar weights are given  by $h(0)=1$ and
        \[h(n)=\frac{(2n+2\eta +2\mu +1)(2\eta+1)_n}{(2\eta +2\nu +1)n!}\left( \sum_{k=0}^n
          \left(
     \begin{array}{c}
        n  \\
        k \\
     \end{array}
         \right)
      \frac{(2\mu)^k}{(2\eta+1)_k}
      \right)^2,\]
          and the Plancherel measure with the support
          $\mathcal{S}\cong[-1,1]$  is given by $d\pi(x)=A(x)dx$ where
          $A(\cos t)=(\sin t)^{2\eta}|\Gamma(\eta+\frac{1}{2}+i\mu \cot(t))|^2
          \exp ((2t-\pi)\mu \cot(t))$,  $0\leq t\leq \pi.$
      Given  $x\in (-1,1)$, since $\pi(\{\alpha_x\})=0$ and    $\alpha_x\in C_0(\NN_0)$,
      by Theorem \ref{main.2} we see  that $\NN_0$ is $\alpha_x$-amenable if and only if $\alpha_x=1$.

 \item[(V)]
         {\bf{Generalized  Soradi hypergroups}}: These are polynomial hypergroups
        of type [V] on $\NN_0$ \cite{BloHey94} with the characters
         \begin{align}\notag
             P_n(\cos \theta)= \frac{\sin (n+1)\theta- k\sin
             n\theta}{(nk+n+1)\sin \theta} \hspace{.5cm}(n\geq 1),
        \end{align}
         and  the density of the Plancherel measure on the dual space
         $\widehat\NN_0\cong [-1,1]$ is given by
         $p(x):= \frac{2(1-x^2)^{1/2}}{\pi(1+k^2-2kx)} (k>1).$
         For $x\in[-1,  1)$, since $\pi(\{\alpha_x\})=0$ and
         $\alpha_x\in C_0(\NN_0)$,
         Theorem \ref{main.2} implies that
         $\NN_0$ is $\alpha_x$-amenable if and only if $\alpha_x=1$.

\item[(VI)] {\bf{Hypergroups associated with infinite distance-transitive graphs}}:
               They are   polynomial  hypergroups on $\NN_0$ depending on $a, b\in\RR$ with $a, b\geq
               2$; and, one can  associate them  with infinite distance-transitive graphs if  $a, b$ are integers.
               These hypergroups have been thoroughly studied by M.
               Voit \cite{Voit.transitive}.
               For   $b>a\geq 2$  (see below) they provide a rare and interesting case of
               $\alpha$-amenability of hypergroups. Their  Haar weights and characters are given by
  \begin{align}\notag
                  h^{(a,b)}(0):=1, \;  h^{(a,b)}(n)=
                  a (a-1)^{n-1} (b-1)^n\quad (n\geq 1),
  \end{align}
    and
      \begin{align}\notag
            P_n^{(a,b)}(x) =\frac{a-1}{a\left( (a-1)(b-1)\right)^{n/2}} \left(
            U_n(x)+\frac{b-2}{\left( (a-1)(b-1)\right)^{1/2}}
            U_{n-1}(x)-\frac{1}{a-1}U_{n-2}(x) \right),
      \end{align}
          respectively, where $U_n(\cos t)=\frac{\sin (n+1)t}{\sin t}$ are the
          Tchebychev  polynomials of the second kind and
          $u_{-1}=u_{-2}:=0$. The dual space $\widehat{\NN_0^{(a,b)}}$
          can be identified with $[-s_1,s_1]$, where $s_1:= \frac{ab-a-b+1}{2\sqrt{(a-1)(b-1)}}$.
          The normalized orthogonality measure $\pi\in M^1(\RR)$
          is
\begin{align}\notag
        d\pi(x)=A(x)dx|_{[-1,1]} \quad\quad\hspace{1.4cm} \text{for}\;
        a\geq b\geq 2,
\end{align}
 and
 \begin{align}%\label{b>a}
           d\pi(x)=A(x)dx|_{[-1,1]} +
           \frac{b-a}{b}ds_0 \quad \text{for}\; b>a\geq 2\notag
 \end{align}
          with $A(x):= \frac{a}{2\pi}\frac{(1-x^2)^{1/2}}{(s_1-x)(x-s_0)}, \; s_0=\frac{2-a-b}{2\sqrt{(a-1)(b-1)}}$.
 Note that
                 \[P_n^{(a,b)} (s_1)=1\hspace{.5cm} \mbox{ and } \hspace{.5cm} P_n^{(a,b)}(s_0)=
                 (1-b)^{-n} \hspace{.5cm} \mbox{ for } n\geq 0.\]

\begin{proposition}\label{last.pro.}
\emph{Let $\NN_0^{(a,b)}$ denote the above  hypergroup. Then }
\begin{itemize}
\item[\emph{(i)}] \emph{for $a\geq b\geq 2$,  $\NN_0^{(a,b)}$is $\alpha_x$-amenable
if and only if \; $x=s_1$. }
\item[\emph{(ii)}] \emph{for $b>a\geq 2$,  $\NN_0^{(a,b)}$is $\alpha_x$-amenable if and only if\;
  $x=s_1$ or $x=s_0$. }
\end{itemize}
\end{proposition}

\begin{proof}
(i) If   $ x\in (-s_1, s_1)$, then    $\pi(\{\alpha_x\})=0$ and
$\alpha_x\in C_0(\NN_0^{(a, b)})$. So,   applying Theorem
\ref{main.2} yields  that   $\NN_0^{(a, b)}$ is
 $\alpha_x$- amenable if and only if $x=s_0$, as $\alpha_{s_0}=1$. \\
(ii) As in part (i), we  can show that if $x\not=s_0$,  then
$\NN_0^{(a, b)}$ is $\alpha_x$- amenable if and only if $x=s_1$. In
the case of $x=s_0$,  obviously  $\alpha_{s_0}\in \ell^1(\NN_0^{(a,
b)})$
  (see also \cite[Remark 1.1]{Voit.transitive}) which implies,   by  Remark \ref{cor.1}
  (ii),  that
 $\NN_0^{(a, b)}$ is $\alpha_{s_0}$-amenable.
\end{proof}

\begin{remark}\label{last.remark}
\emph{Notice that in the previous example $\widehat{K}$  contains
      two positive characters $\alpha_{s_0}$ and $\alpha_{s_1}$ with  diverse behaviours. Indeed, Part (i) shows that $ \NN_0^{(a,b)}$
      is $\alpha_{s_1}$-amenable but not $\alpha_{s_0}$-amenable if $a\geq b \geq 2$,
      whereas  Part (ii) shows  that
      $ \NN_0^{(a,b)}$ is  $\alpha_{s_1}$ and  $\alpha_{s_0}$-amenable for $b>a\geq 2$.
      In latter case,  the functional $m_{\alpha_{s_0}}$, given
      in Remark \ref{cor.1} (ii), is a unique $\alpha_{s_0}$-mean on
      $\ell^\infty(\NN_0^{(a, b)})$ while
      the cardinality of $\alpha_{s_1}$-means on $\ell^\infty(\NN_0^{(a, b)})$
      is infinity; see \cite{azimi.C.Math.Rep., Ska92}. }
\end{remark}

\end{enumerate}


\begin{thebibliography}{10}


\bibitem{thesis}
A. Azimifard,
\newblock { \em \emph{ $\alpha$-Amenability of Banach algebras on commutative
  hypergroups}}.
\newblock Dissertation, Technische Universität München, 2006.
\vspace{-2mm}


\bibitem{azimifard.Monath}
A. Azimifard,
\newblock{ On  the $\alpha$-Amenability of Hypergroups.}  Montash. Math.
(to appear 2008)
\vspace{-2mm}

\bibitem{azimi.C.Math.Rep.}
A. Azimifard,
\newblock{ Hypergroups with the unique $\alpha$-means. }  C. R. Math. Acad. Sci. Soc. R. Can. (to appear
2008)



\vspace{-2mm}
\bibitem{BloHey94}
W. R. Bloom  and H.  Heyer,
\newblock { \em \emph{Harmonic Analysis of Probability Measures on  Hypergroups. }}
\newblock De Gruyter, 1994.

\vspace{-2mm}


\bibitem{BonDun73}
F. F. Bonsall  and J.  Duncan,
\newblock { \em \emph{Complete Normed Algebras. }}
\newblock Springer-Verlag, New York-Heidelberg, 1973.


\vspace{-2mm}




\bibitem{Dunford}
N. Dunford  and J. T.  Schwartz,
\newblock {\em \emph{ Linear Operators I.}}
\newblock Wiley \& Sons, 1988.

\vspace{-2mm}

\bibitem{f.l.s}
F.  Filbir, R. Lasser and   R. Szwarc,
\newblock Reiter's condition {$P_1$} and approximate identities for
  hypergroups.
\newblock {\em \emph{Monat. Math. }} 143 (2004) 189--203.

\vspace{-2mm}


\bibitem{Frankcompact}
F.  Filbir, R. Lasser and   R. Szwarc,
\newblock Hypergroups of compact type.
\newblock {\em \emph{J. Comp. and App. Math.}} 178 (1) (2005) 205--214.

\vspace{-2mm}

\bibitem{ismail}
M. Ismail,
\newblock {\em \emph{ Classical and quantum orthogonal  polynomials in one variables.}}
\newblock Cambridge University  Press, 2005.
%\\[-1cm]

\vspace{-2mm}
\bibitem{Jew75}
R. I.  Jewett,
\newblock Spaces with an abstract convolution of measures.
\newblock {\em \emph{Adv. in Math.}}  18 (1975) 1--101.
\vspace{-2mm}

 \bibitem{Johnson}
B. E. Johnson,
\newblock Cohomology in {Banach} algebras.
\newblock {Mem. Amer. Math. Soc. 127 (1972).}
\vspace{-2mm}

\bibitem{Koornwinder.1}
T. H. Koornwinder,
\newblock  Two-variable analogues of the classical orthogonal polynomials. In:
 Theory And Application Of Special Functions, (R. Askey, ed.), 
pp. 435 -- 495, Academic Press, New York, 1975. 
             
\vspace{-2mm}


\bibitem{lasser.lagure}
R.  Lasser,
\newblock Orthogonal polynomials and hypergroups II-the symmetric
case.
\newblock {\em \emph{Trans. Amer. Math. Soc.}} 341 (1994) 749--770.
%\\[-1cm]
\vspace{-2mm}


\vspace{-2mm}
\bibitem{Nev79}
 P. G. Nevai,
\newblock {\em \emph{ Orthogonal Polynomials}}.
\newblock Mem.  Amer. Math. Soc. 213, 1979.


\vspace{-2mm}
\bibitem{Ska92}
M. Skantharajah,
\newblock Amenable hypergroups.
\newblock {\em \emph{Illinois J. Math.}} 36 (1) (1992) 15--46.
\vspace{-2mm}

\bibitem{Spec75}
R. Spector,
\newblock Aper\c{c}u de la th\'{e}orie des hypergroupes.
\newblock In {\em \emph{Anal. harmon. Groupes de Lie}}, {\em
  \emph{LNM  497 }} (1975) 643--673.
%\\[-1cm]

\vspace{-2mm}
  \bibitem{voit.p.c.}
M. Voit,
\newblock Positive characters on commutative hypergroups and some
applications.
\newblock {\em \emph{Math. Z.}} 198 (1988) 405--421.
\vspace{-2mm}

\bibitem{Voit91}
M. Voit,
\newblock On the dual space of a commutative hypergroup.
\newblock {\em \emph{Arch. Math.}} 56 (4) (1991) 380--385.
\vspace{-2mm}


    
\bibitem{voit.sy}
M. Voit,
\newblock Factorization of probability measures on symmetric hypergroups.
\newblock  J. Aust. Math. Soc.   (Series A)  50   (1991) 417 -- 467.
\vspace{-2mm}



\bibitem{Voit.transitive}
M. Voit,
\newblock  A product formula for orthogonal polynomials
associated with infinite distance-transitive graphs.
\newblock {\em \emph{J. App. Theory. }}120 (2003) 337--354.
\vspace{-2mm}

\bibitem{vogel.spectral}
M. Vogel,
\newblock Spectral synthesis on algebras of orthogonal polynomial
series.
\newblock {\em \emph{Math. Z.}} 194 (1) (1987) 99--116.

\vspace{-2mm}
\bibitem{vrem.joint}
R. C.  Vrem,
\newblock Hypergroups joins and their dual objects.
\newblock {\em \emph{Pacific J. Math.}} 111 (1984) 483-- 495.
\vspace{-2mm}

\bibitem{Wol84}
 S. Wolfenstetter,
\newblock {\em \emph{{J}acobi-{P}olynome und {B}essel-{F}unktionen unter dem
  {G}esichtspunkt der harmonischen {A}nalyse}}.
\newblock Dissertation , Technische Universit\"at M\"unchen,  1984.

\vspace{-2mm}
\bibitem{zeuner.several}
H. Zeuner,
\newblock Polynomial hypergroups in several variables.
\newblock {\em \emph{Arch. Math.}} 58  (1992) 425--434.
%\\[-.5cm]


\end{thebibliography}
 \end{document}